\documentclass[reqno,11pt]{amsart}
\usepackage{amscd,amssymb,verbatim}

\numberwithin{equation}{section} \theoremstyle{plain}
\theoremstyle{plain}
\newtheorem{Thm}[subsection]{Theorem}
\newtheorem{Cor}[subsection]{Corollary}
\newtheorem{Lem}[subsection]{Lemma}
\newtheorem{Prop}[subsection]{Proposition}

\theoremstyle{definition}
\newtheorem{Def}[subsection]{Definition}

\theoremstyle{remark}

\newtheorem{rem}[subsection]{Remark}

\newtheorem{exs}{Examples}

\newenvironment{thm}%
          { \begin{Thm}  }%
          { \end{Thm} }

\newenvironment{lem}%
          { \begin{Lem}    }%
          { \end{Lem} }

          { \begin{Prop}  }%
          { \end{Prop} }

          { \begin{Cor} }%
          { \end{Cor} }

          { \begin{Def} }%
          { \end{Def} }

\newcommand{\ecomp}{C_{c}(E_{s})}
\newcommand{\vcomp}{C_{c}(V)}
\newcommand{\lw}{{\ell_{w}^2}(V)}
\newcommand{\la}{{\ell_{a}^2}{(E_{s})}}
\newcommand{\delswa}{\Delta_{\sigma}}
\newcommand{\hmax}{H_{\max}}
\newcommand{\hmin}{H_{\min}}
\newcommand{\Dom}{\operatorname{Dom}}

\title{Essential self-adjointness of magnetic Schr\"odinger operators on locally finite graphs}
\author{Ognjen Milatovic}
\address{Department of Mathematics and Statistics\\
           University of North Florida   \\
           Jacksonville, FL 32224 \\
           USA
            }
\email{omilatov@unf.edu} 
\begin{document}
\maketitle
\begin{abstract} We give sufficient conditions for essential self-adjointness of magnetic Schr\"odinger operators on locally finite graphs. Two of the main theorems of the present paper generalize recent results of Torki-Hamza.
\end{abstract}
\section{Introduction and the Main Results}\label{S:main}
\subsection{The setting}\label{SS:setting}
Let $G=(V,E)$ be an infinite graph without loops and multiple edges between vertices. By $V=V(G)$ and $E=E(G)$ we denote the set of vertices  and the set of unoriented edges of $G$ respectively. In what follows, the notation $m(x)$ indicates the degree of a vertex $x$, that is, the number of edges that meet at $x$. We assume that $G$ is locally finite, that is, $m(x)$ is finite for all $x\in V$.

In what follows, $x\sim y$ indicates that there is an edge that connects $x$ and $y$. We will also need a set of oriented edges
\begin{equation}\label{E: edge-or}
E_{0}:=\{[x,y],[y,x]: x,\,y\in V\textrm{ and } x\sim y\}.
\end{equation}
The notation $e=[x,y]$ indicates an oriented edge $e$ with starting vertex $o(e)=x$ and terminal vertex $t(e)=y$. The definition~(\ref{E: edge-or}) means that every unoriented edge in $E$ is represented by two oriented edges in $E_0$. Thus, there is a two-to-one map $p\colon E_0\to E$. For $e=[x,y]\in E_{0}$, we denote the corresponding reverse edge by $\widehat{e}=[y,x]$. This gives rise to an involution $e\mapsto \widehat{e}$ on $E_{0}$.

To help us write formulas in unambiguous way, we fix an orientation on each edge by specifying a subset $E_{s}$ of $E_0$ such that $E_0=E_{s}\cup \widehat{E_{s}}$ (disjoint union), where $\widehat{E_{s}}$ denotes the image of $E_{s}$ under the involution $e\mapsto \widehat{e}$. Thus, we may identify $E_{s}$ with $E$ by the map $p$.

In the sequel, we assume that $G$ is connected, that is, for any $x,\,y\in V$ there exists a path $\gamma$ joining $x$ and $y$. Here, $\gamma$ is a sequence $x_1,\,x_2,\,\dots,x_n\in V$ such that $x=x_1$, $y=x_n$, and $x_{j}\sim x_{j+1}$ for all $1\leq j\leq n-1$. The length of a path $\gamma$ is defined as the number of edges in $\gamma$.

The distance $d(x,y)$ between vertices $x$ and $y$ of $G$ is defined as the number of edges in the shortest path connecting the vertices $x$ and $y$. Fix a vertex $x_0\in V$ and define $r(x):=d(x_{0}, x)$. The $n$-neighborhood $B_{n}(x_0)$ of $x_0\in V$ is defined as
\begin{equation}\label{E: neighb-n}
\{x\in V\colon r(x)\leq n\}\cup\{e=[x,y]\in E_{s}\colon r(x)\leq n\textrm{ and }r(y)\leq n\}.
\end{equation}

In what follows, $C(V)$ is the set of complex-valued functions on $V$, and $C(E_{s})$ is the set of functions $Y\colon E_{0}\to\mathbb{C}$ such that $Y(e)=-Y(\widehat{e})$. The notations $\vcomp$ and $\ecomp$ denote the sets of finitely supported elements of $C(V)$ and $C(E_{s})$ respectively.

In the sequel, we assume that $V$ is equipped with a weight $w\colon V\to\mathbb{R}^{+}$. By $\lw$ we denote the space of functions $f\in C(V)$ such that $\|f\|<\infty$, where $\|f\|$ is the norm corresponding to the inner product
\begin{equation}\label{E:inner-w}
(f,g):=\sum_{x\in V}w(x)f(x)\overline{g(x)}.
\end{equation}
Additionally, we assume that $E$ is equipped with a weight $a\colon E_{0}\to\mathbb{R}^{+}$ such that $a(e)=a(\widehat{e})$ for all $e\in E_{0}$. This makes $G=(G, w, a)$ a weighted graph weights $w$ and $a$.

\subsection{Magnetic Schr\"odinger Operator}\label{SS:magnetic-schro}
Let $U(1):=\{z\in\mathbb{C}\colon |z|=1\}$ and $\sigma\colon E_{0}\to U(1)$ with  $\sigma(\widehat{e})=\overline{\sigma(e)}$ for all $e\in  E_{0}$, where $\overline{z}$ denotes the complex conjugate of $z\in \mathbb{C}$.

We define the magnetic Laplacian $\delswa\colon C(V)\to C(V)$ on the graph $(G, w, a)$ by the formula
\begin{equation}\label{E:magnetic-lap}
    (\delswa u)(x)=\frac{1}{w(x)}\sum_{e\in\mathcal{O}_{x}}a(e)(u(x)-\sigma(\widehat{e})u(t(e))),
\end{equation}
where $x\in V$ and
\begin{equation}\label{E:neighborhood-o-x}
\mathcal{O}_{x}:=\{e\in E_0\colon o(e)=x\}.
\end{equation}

For the case $a\equiv 1$ and $w\equiv 1$, the definition~(\ref{E:magnetic-lap}) is the same as in Dodziuk--Mathai~\cite{Dodziuk-Mathai-03}. For the case $\sigma\equiv 1$, see Sy--Sunada~\cite{Sunada-92} and Torki--Hamza~\cite{Torki-10}.

Let $q\colon V\to \mathbb{R}$, and consider a Schr\"odinger-type expression
\begin{equation}\label{E:magnetic-schro}
Hu:= \delswa u +qu
\end{equation}

We give sufficient conditions for $H|_{\vcomp}$ to be essentially self-adjoint in the space $\lw$. We first state the main results, and in Section~\ref{SS:problem-bacground}  we make a few remarks concerning the existing work on the essential self-adjointness problem on locally finite graphs.

\begin{thm}\label{T:main-0} Assume that $(G, w, a)$ is an infinite, locally finite, connected, oriented, weighted graph with $w(x)\equiv w_0$, where $w_0>0$ is a constant.  Additionally, assume that there exists a constant $C\in\mathbb{R}$ such that $q(x)\geq -C$ for all $x\in V$. Then, the operator $H|_{\vcomp}$ is essentially self-adjoint in $\lw$.
\end{thm}

In the next theorem, we will need the following additional assumption on the graph $G$.

\bigskip

\noindent{\textbf{Assumption (A)} Assume that
\begin{equation}\label{E:bounded-assumption}
    \lim_{n\to\infty}\frac{m_{n}a_{n}}{n^2}=0,
\end{equation}
where
\begin{equation}\label{E:bounded-assumption-2}
m_{n}:=\max_{x\in B_n(x_0)}(m(x))\qquad\textrm{and}\qquad a_{n}:=\max_{x\in B_n(x_0)}\left(\max_{e\sim x}\left(\frac{a(e)}{w(x)}\right)\right),
\end{equation}
where $B_n(x_0)$ as in~(\ref{E: neighb-n}), and $e\sim x$, with $e\in E_{s}$ and $x\in V$, indicates that $t(e)=x$ or $o(e)=x$.

\begin{thm}\label{T:main-1} Assume that $(G, w, a)$ is an infinite, locally finite, connected, oriented, weighted graph. Assume that the Assumption (A) is satisfied. Additionally, assume that there exists a constant $C\in\mathbb{R}$ such that
\begin{equation}\label{E:assumption-semibounded}
    (Hu,u)\geq -C\|u\|^2\qquad\textrm{ for all }u\in\vcomp,
\end{equation}
where $(\cdot,\cdot)$ and $\|\cdot\|$ are as in~(\ref{E:inner-w}). Then, the operator $H|_{\vcomp}$ is essentially self-adjoint in $\lw$.
\end{thm}

In the next theorem, we will need the notion of weighted distance on $G$. Let $a\colon E_{0}\to\mathbb{R}^{+}$ be as in~(\ref{E:magnetic-lap}). Following Colin de Verdi\`ere, Torki-Hamza, and Truc~\cite{vtt-10}, we define the weighted distance $d_{w,a}$ on $G$ as follows:
\begin{equation}\label{E:w-a-dist}
d_{w,a}(x,y):=\inf_{\gamma\in P_{x,y}}L_{w,a}(\gamma),
\end{equation}
where $P_{x,y}$ is the set of all paths $\gamma\colon x=x_1,\,x_2,\,\dots,x_n=y$ such that $x_{j}\sim x_{j+1}$ for all $1\leq j\leq n-1$, and the length $L_{w,a}(\gamma)$ is computed as follows:
\begin{equation}\label{E:length-a-dist}
L_{w,a}(\gamma)=\sum_{j=1}^{n-1}\sqrt{\frac{\min\{w(x_j), w(x_{j+1})\}}{a([x_{j},x_{j+1}])}}.
\end{equation}
We say that the metric space $(G,d_{w,a})$ is complete if every Cauchy sequence of vertices has a limit in $V$.

In what follows, we say that $G$ is a graph of bounded degree if there exists a constant $N>0$ such that $m(x)\leq N$ for all $x\in V$.

\begin{thm}\label{T:main-2} Assume that $(G, w, a)$ is an infinite, locally finite, connected, oriented, weighted graph. Assume that $G$ is a graph of bounded degree. Assume that $(G,d_{w,a})$ is a complete metric space. Additionally, assume that $H$ satisfies~(\ref{E:assumption-semibounded}). Then, the operator $H|_{\vcomp}$ is essentially self-adjoint in $\lw$.
\end{thm}

\begin{rem} Let $d_{w,a}$ be as in~(\ref{E:w-a-dist}). It is easily seen that if $G$ is a graph of bounded degree and if~(\ref{E:bounded-assumption}) is satisfied, then $(G,d_{w,a})$ is complete.

The following example describes a graph $G$ of bounded degree such that $(G,d_{w,a})$ is complete and~(\ref{E:bounded-assumption}) is not satisfied.

\begin{exs} (i) Denote $\mathbb{Z}_{+}:=\{1,2,3,\dots\}$, and consider the graph $G_1=(V,E)$ with $V=\mathbb{Z}_{+}\cup\{0\}$ and $E=\{[n-1,n]\colon n\in \mathbb{Z}_{+}\}$.  Define $a([n-1,n])=n$ and $w(n-1)=\frac{1}{n}$, for all $n\in \mathbb{Z}_{+}$.

Since $w(x)$ is not constant, we cannot use Theorem~\ref{T:main-0} in this example.

Let $K\in \mathbb{Z}_{+}$ and let $m_K$ and $a_K$ be as in~(\ref{E:bounded-assumption-2}) with $n=K$ and $x_0=0$. We have $m_K=2$ and
$\displaystyle a_{K}=(K+1)^2$. Thus,
\[
\lim_{K\to\infty}\frac{m_{K}a_{K}}{K^2}=2,
\]
and~(\ref{E:bounded-assumption}) is not satisfied. Thus, in this example, we cannot use Theorem~\ref{T:main-1}.

Fix $K_0\in \mathbb{Z}_{+}\cup\{0\}$, and let $K>K_0$. For $x_0=K_0$ and $x=K$, by~(\ref{E:w-a-dist}) we have
\[
d_{w,a}(x_0,x)=\sum_{n=K_0}^{K-1}\frac{1}{\sqrt{(n+1)(n+2)}}\to\infty,\quad\textrm{as }K\to\infty.
\]
Thus, the metric $d_{w,a}$ is complete. Additionally, the graph $G_1$ has bounded degree. By Theorem~\ref{T:main-2} the operator $\Delta_{\sigma}|_{\vcomp}$ is essentially self-adjoint in $\lw$.

\medskip

The following example describes a graph of unbounded degree such that~(\ref{E:bounded-assumption}) is satisfied.

\medskip

\noindent (ii) Consider $G_2=(V,E)$, where $V=\{x_0,x_1,x_2,\dots\}$. The vertices are arranged in a ``triangular" pattern so that $x_0$ is in the first row, $x_1$ and $x_2$ are in the second row, $x_3$, $x_4$, and $x_5$ are in the third row, and so on. The vertex $x_0$ is connected to $x_1$ and $x_2$. The vertex $x_i$, where $i=1,2$, is connected to every vertex $x_j$, where $j=3,4,5$. The pattern continues so that each of $k$ vertices in the $k$-th row is connected to each of $k+1$ vertices in the $(k+1)$-th row. Define $a(e)\equiv 1$ for all $e\in E$. For every vertex $x$ in the $n$-th row, define $w(x)=n^{-1/2}$.

Since $w(x)$ is not constant, we cannot use Theorem~\ref{T:main-0}. Since $G_2$ does not have a bounded degree, we cannot use Theorem~\ref{T:main-2}.

Let $K\in\{1,2,\dots\}$. Let $m_K$ and $a_K$ be as in~(\ref{E:bounded-assumption-2}) with $n=K$ and $x_0$ as in this example. We have $m_K=2K+2$ and
$\displaystyle a_{K}=\sqrt{K+1}$. Thus,
\[
\lim_{K\to\infty}\frac{m_{K}a_{K}}{K^2}=0,
\]
and~(\ref{E:bounded-assumption}) is satisfied. By Theorem~\ref{T:main-1} the operator $\Delta_{\sigma}|_{\vcomp}$ is essentially self-adjoint in $\lw$.
\end{exs}
\end{rem}

\begin{rem} In the context of a not necessarily complete graph of bounded degree, a sufficient condition for essential-self adjointness of $\Delta_{\sigma}|_{\vcomp}$ in $\lw$ is given by Colin de Verdi\`ere, Torki-Hamza, and Truc~\cite[Theorem 3.1]{vtt-10-preprint}. In the case $q\equiv 0$, Theorem~\ref{T:main-2} is contained in~\cite[Theorem 3.1]{vtt-10-preprint}.


\end{rem}

\section{Background of the Problem}\label{SS:problem-bacground}
In the context of a locally finite graph $G=(V,E)$, recently there has been a lot of interest in the operator
\begin{equation}\label{E:ord-lap}
    (\Delta u)(x)=\frac{1}{w(x)}\sum_{e\in\mathcal{O}_{x}}a(e)(u(x)-u(t(e))),
\end{equation}
where $x\in V$ and $\mathcal{O}_{x}$ is as in~(\ref{E:neighborhood-o-x}).

In many spectral-theoretic investigations of $\Delta$ and $\Delta+q$, where $q\colon V\to \mathbb{R}$ is a real-valued function, it is helpful to have a self-adjoint operator. Thus, finding sufficient conditions for essential self-adjointness of $\Delta$ and $\Delta+q$ is an important problem in analysis on locally finite graphs.  Note that $\Delta$ in~(\ref{E:ord-lap}), also known as physical Laplacian, is generally an unbounded operator in $\lw$. Putting $w\equiv 1$ and $a\equiv 1$ in~(\ref{E:ord-lap}) and dividing by the degree function $m(x)$, we get the normalized Laplacian, which is a bounded operator on $\lw$, with inner product as in~(\ref{E:inner-w}) with  $w(x)=m(x)$. The normalized Laplacian has been studied extensively; see, for instance, Chung~\cite{Chung} and Mohar--Woess~\cite{Mohar-89}.

In the discussion that follows, the local finiteness assumption is understood, unless specified otherwise. The essential self-adjointness of~$\Delta|_{\vcomp}$,  where $\Delta$ is as in~(\ref{E:ord-lap}) with $w\equiv 1$ and $a\equiv 1$, was proven by Wojciechowski~\cite{Woj-08} and Weber~\cite{Weber-10}.  For $\Delta$ is as in~(\ref{E:ord-lap}) with $w\equiv 1$, the essential self-adjointness of~$\Delta|_{\vcomp}$ was proven by Jorgensen~\cite{Jor-08} (see also Jorgensen--Pearse~\cite{Jor-08-preprint}). With regard to Theorem~\ref{T:main-0} of the present paper, Torki-Hamza~\cite{Torki-10} proved the essential self-adjointness of $(\Delta+q)|_{\vcomp}$, where $\Delta$ is as in~(\ref{E:ord-lap}) with $w\equiv c_0$ and $q\geq-c_1$, where $c_0>0$ and $c_1\in\mathbb{R}$ are constants. The results of Wojciechowski~\cite{Woj-08}, Weber~\cite{Weber-10}, and Jorgensen~\cite{Jor-08} on the essential self-adjointness of $\Delta$ and the result of Torki-Hamza~\cite{Torki-10} on the essential self-adjointness of $(\Delta+q)|_{\vcomp}$ with $q\geq -c_1$, where $c_1$ is a constant, are all contained in Keller--Lenz~\cite{Keller-Lenz-09} and Keller--Lenz~\cite{Keller-Lenz-10}.

Under the assumption~(\ref{E:bounded-assumption}) above, the essential self-adjointness of  $(d\delta+\delta d)|_{\Omega_{0}(G)}$, where $\Omega_{0}(G)$ denotes finitely supported forms $\alpha\in C(V)\oplus C(E)$, was proven by Masamune~\cite{Masamune-09}. Additionally, Masamune~\cite{Masamune-09} studied $L^p$-Liouville property for non-negative subharmonic forms on $G$.

In the context of a graph of bounded degree, Torki-Hamza~\cite{Torki-10} made an important link between the essential self-adjointness of  $(\Delta+q)|_{\vcomp}$, where $\Delta$ is as in~(\ref{E:ord-lap}) with $w\equiv 1$, and completeness of the weighted metric $d_{1,a}$ in~(\ref{E:w-a-dist}) above; namely, if $d_{1,a}$ is complete and if $(\Delta+q)|_{\vcomp}$ is semi-bounded below, then $(\Delta+q)|_{\vcomp}$ is essentially self-adjoint on the space $\lw$ with $w\equiv 1$. Theorem~\ref{T:main-2} of the present paper extends this result to the operator~(\ref{E:magnetic-schro}).

For a study of essential self-adjointness of $(\Delta+q)|_{\vcomp}$ on a metrically non-complete graph, see Colin de Verdi\`ere, Torki-Hamza, and Truc~\cite{vtt-10}.  Adjacency matrix operator on a locally finite graph was studied in Gol\'enia~\cite{Golenia-10}. For a study of the problem of deficiency indices for Schr\"odinger operators on a locally finite graph, see~Gol\'enia--Schumacher~\cite{Golenia-Schu-10}.


Kato's inequality for $\Delta_{\sigma}$ as in~(\ref{E:magnetic-lap}), with $w\equiv 1$ and $a\equiv 1$, was proven in Dodziuk--Mathai~\cite{Dodziuk-Mathai-03} and used to study asymptotic properties of the spectrum of a certain discrete magnetic Schr\"odinger operator. For a study of essential self-adjointness of the magnetic Laplacian on a metrically non-complete graph, see Colin de Verdi\`ere, Torki-Hamza, and Truc~\cite{vtt-10-preprint}.  A different model for discrete magnetic Laplacian was given by Sushch~\cite{Sushch-10}. In the model of~\cite{Sushch-10}, the essential self-adjointness  of a semi-bounded below discrete magnetic Schr\"odinger operator was proven.

Dodziuk~\cite{Dodziuk-06}, Wojciechowski~\cite{Woj-08}, Wojciechowski~\cite{Woj-09}, and Weber~\cite{Weber-10}  explored connections between stochastic completeness and the essential self-adjointness of $\Delta$. For extensions to the more general context of Dirichlet forms on discrete sets, see Keller--Lenz~\cite{Keller-Lenz-09} and Keller--Lenz~\cite{Keller-Lenz-10}. For a related study of random walks on infinite graphs, see Dodziuk~\cite{Dodziuk-84}, Dodziuk-Karp~\cite{Dodziuk-88}, Woess~\cite{Woess-00}, and references therein.

Finally, we remark that the problem of essential self-adjointness of Schr\"odinger operators on infinite graphs has a strong connection to the corresponding problem on non-compact Riemannian manifolds; see Gaffney~\cite{ga}, Oleinik~\cite{ol}, Oleinik~\cite{Oleinik94}, Braverman~\cite{br}, Shubin~\cite{sh1}, Shubin~\cite{sh01}, and~\cite{bms}.

\section{Preliminaries}
In what follows, $d\colon C(V)\to C(E_{s})$ is the standard differential
\[
du(e):=u(t(e))-u(o(e)).
\]
The deformed differential $d_{\sigma}\colon C(V)\to C(E_{s})$ is defined as
\begin{equation}\label{E:d-sigma}
(d_{\sigma}u)(e):=\overline{\sigma(e)}u(t(e))-u(o(e)),\qquad \textrm{for all }u\in C(V),
\end{equation}
where $\sigma$ is as in~(\ref{E:magnetic-lap}).

The deformed co-differential $\delta_{\sigma}\colon C(E_{s})\to C(V)$ is defined as follows:
\begin{equation}\label{E:delta-sigma}
(\delta_{\sigma}Y)(x):=\frac{1}{w(x)}\sum_{\substack{e\in E_{s}\\t(e)=x}}\sigma(e)a(e)Y(e)-\frac{1}{w(x)}\sum_{\substack{e\in E_{s}\\o(e)=x}}a(e)Y(e), \end{equation}
for all $Y\in C(E_{s})$, where $\sigma$, $w$, and $a$ are as in~(\ref{E:magnetic-lap}).

Let $\la$ denote the space of functions $F\in C(E_{s})$ such that $\|F\|<\infty$, where $\|F\|$ is the norm corresponding to the inner product
\[
(F,G):=\sum_{e\in E_{s}}a(e)F(e)\overline{G(e)}.
\]
For a general background on the theory of magnetic Laplacian on graphs, see~Mathai--Yates~\cite{Mathai-Yates} and Sunada~\cite{Sunada-94}.
\begin{lem} The following equality holds:
\begin{equation}\label{E:adjoint-delta}
(d_{\sigma}u,Y)=(u,\delta_{\sigma}Y),\qquad\textrm{for all }u\in \lw,\,\,Y\in\ecomp,
\end{equation}
where $(\cdot,\cdot)$ on the left-hand side (right-hand side) denotes the inner product in $\la$ (in $\lw$).
\end{lem}
\begin{proof} Using~(\ref{E:d-sigma}) and~(\ref{E:delta-sigma}) we have
\begin{align}
(u,\delta_{\sigma}Y) &=\sum_{x\in V}u(x)\left(\sum_{\substack{e\in E_{s}\\t(e)=x}}a(e)\overline{\sigma(e)Y(e)}-\sum_{\substack{e\in E_{s}\\o(e)=x}}a(e)\overline{Y(e)}\right)\nonumber\\
&=\sum_{e\in E_{s}}a(e)u(t(e))\overline{\sigma(e)Y(e)}-\sum_{e\in E_{s}}a(e)u(o(e))\overline{Y(e)}\nonumber\\
&=\sum_{e\in E_{s}}a(e)(\overline{\sigma(e)}u(t(e))-u(o(e)))\overline{Y(e)}=(d_{\sigma}u,Y).\nonumber
\end{align}
The convergence of the sums is justified by observing that only finitely many $x\in V$ contribute to the sum as $Y$ has finite support.
\end{proof}

Using the definitions~(\ref{E:d-sigma}) and~(\ref{E:delta-sigma}) together with the properties $a(\widehat{e})=a(e)$, $\sigma(\widehat{e})=\overline{\sigma(e)}$, and $|\sigma(e)|=1$, which hold for all $e\in E_{0}$, one can easily prove the following lemma.

\begin{lem}\label{L:laplacian} The equality $\delta_{\sigma}d_{\sigma}u=\Delta_{\sigma}u$ holds for all $u\in C(V)$.
\end{lem}

The following lemma follows easily from Lemma~\ref{L:laplacian} and~(\ref{E:adjoint-delta}).

\begin{lem}\label{L:symmetric} The operator $\Delta_{\sigma}|_{\vcomp}$ is symmetric in $\lw$:
\[
(\Delta_{\sigma}u,v)=(u, \Delta_{\sigma}v), \quad \textrm{for all }u,\,v\in\vcomp.
\]
\end{lem}

\begin{lem}\label{L:leibniz} For all $u\,,v\in C(V)$ the following property holds:
\begin{align}\label{E:laplacian-leibniz}
(\Delta_{\sigma}(uv))(x)&=(\Delta_{\sigma}u)(x)v(x)\nonumber\\
&+\frac{1}{w(x)}\sum_{e\in\mathcal{O}_{x}}a(e)\sigma(\widehat{e})u(t(e))(v(x)-v(t(e))),
\end{align}
where $x\in V$ and $\mathcal{O}_{x}$  is as in~(\ref{E:neighborhood-o-x}).
\end{lem}
\begin{proof}  Using the definition~(\ref{E:magnetic-lap}) we have
\begin{align}\label{E:laplacian-product}
(\Delta_{\sigma}(uv))(x)&=\frac{1}{w(x)}\sum_{e\in\mathcal{O}_{x}}a(e)(u(x)v(x))\nonumber\\
&-\frac{1}{w(x)}\sum_{e\in\mathcal{O}_{x}}a(e)\sigma(\widehat{e})u(t(e))v(t(e)).
\end{align}
Adding and subtracting
\[
\frac{1}{w(x)}\sum_{e\in\mathcal{O}_{x}}a(e)\sigma(\widehat{e})u(t(e))v(x)
\]
on the right-hand side of~(\ref{E:laplacian-product}) and grouping the terms appropriately, we get~(\ref{E:laplacian-leibniz}).
\end{proof}

In the proof of the following proposition, we will use a technique similar to Shubin~\cite[Section 5.1]{sh01}, Masamune~\cite{Masamune-09}, and Torki-Hamza~\cite{Torki-10}.

\begin{Prop}\label{P:leibniz} Assume that $u\in\lw$ and $Hu=0$. Then the following holds for all $\phi\in\vcomp$:
\begin{align}\label{E:H-leibniz-3}
&(H(u\phi), u\phi)\nonumber\\
&=\sum_{e\in E_{s}}a(e)\sigma_1(\widehat{e})[u_1(t(e))u_1(o(e))+u_2(t(e))u_2(o(e))](\phi(o(e))-\phi(t(e)))^2\nonumber\\
&+\sum_{e\in E_{s}}a(e)\sigma_2(\widehat{e})[-u_1(o(e))u_2(t(e))+\nonumber\\
&+u_1(t(e))u_2(o(e))](\phi(o(e))-\phi(t(e)))^2,
\end{align}
where $u_1:=\textrm{Re }u$, $u_2:=\textrm{Im }u$, $\sigma_1:=\textrm{Re }\sigma$, and $\sigma_2:=\textrm{Im }\sigma$.
\end{Prop}
\begin{proof}
Using~(\ref{E:laplacian-leibniz}) with $v=\phi$, we obtain
\begin{align}\label{E:H-leibniz}
(H(u\phi))(x)&=(Hu)(x)\phi(x)\nonumber\\
&+\frac{1}{w(x)}\sum_{e\in\mathcal{O}_{x}}a(e)\sigma(\widehat{e})u(t(e))(\phi(x)-\phi(t(e))).
\end{align}
Taking the inner product $(\cdot,\cdot)$ with $u\phi$ on both sides of~(\ref{E:H-leibniz}), we obtain:
\begin{align}\label{E:H-inner}
&(H(u\phi), u\phi)=(\phi(Hu),u\phi)\nonumber\\
&+\sum_{x\in V}\sum_{e\in\mathcal{O}_{x}}a(e)\sigma(\widehat{e})u(t(e))(\phi(x)-\phi(t(e)))\overline{u(x)}\phi(x).
\end{align}
Taking the real parts on both sides of~(\ref{E:H-inner}), we get
\begin{align}\label{E:H-inner-1}
&(H(u\phi), u\phi)=\textrm{Re }(\phi(Hu),u\phi)\nonumber\\
&+\textrm{Re}\Big(\sum_{x\in V}\sum_{e\in\mathcal{O}_{x}}a(e)\sigma(\widehat{e})u(t(e))(\phi(x)-\phi(t(e)))\overline{u(x)}\phi(x)\Big).
\end{align}

Since $\sigma(\widehat{e})=\overline{\sigma(e)}$, it follows that $\sigma_1(\widehat{e})=\sigma_1(e)$ and $\sigma_2(\widehat{e})=-\sigma_2(e)$.
Substituting $u=u_1+iu_2$, $\sigma=\sigma_1+i\sigma_2$ and $Hu=0$ in~(\ref{E:H-inner-1}) leads to
\begin{equation}\label{E:split-1-2}
(H(u\phi), u\phi)=J_1+J_2,
\end{equation}
where
\begin{align*}
J_1:=&\sum_{x\in V}\sum_{e\in\mathcal{O}_{x}}a(e)\sigma_1(\widehat{e})[u_1(t(e))u_1(x)+\\
&+u_2(t(e))u_2(x)](\phi^2(x)-\phi(x)\phi(t(e))),
\end{align*}
and
\begin{align*}
J_2:=&\sum_{x\in V}\sum_{e\in\mathcal{O}_{x}}a(e)\sigma_2(\widehat{e})[-u_1(x)u_2(t(e))+\\
&+u_1(t(e))u_2(x)](\phi^2(x)-\phi(x)\phi(t(e))).
\end{align*}

In each of the sums $J_1$ and $J_2$ an edge $e=[x,y]\in E_{0}$ occurs twice: once as $[x,y]$ and once as $[y,x]$.
Since $a([x,y])=a([y,x])$, $\sigma_1([x,y])=\sigma_1([y,x])$, and $\sigma_2([x,y])=-\sigma_2([y,x])$, it follows that the expressions
\begin{align*}
&a(e)\sigma_1(\widehat{e})(u_1(t(e))u_1(x)+u_2(t(e))u_2(x))\\
&=a([x,y])\sigma_1([y,x])(u_1(y)u_1(x)+u_2(y)u_2(x))
\end{align*}
and
\begin{align*}
&a(e)\sigma_2(\widehat{e})(-u_1(x)u_2(t(e))+u_1(t(e))u_2(x))\\
&=a([x,y])\sigma_2([y,x])(-u_1(x)u_1(y)+u_1(y)u_2(x))
\end{align*}
are invariant under the involution $e\mapsto\widehat{e}$.
Hence, in the sum $J_1$, the contribution of $e=[x,y]$ and $\widehat{e}=[y,x]$ together is
\begin{equation}\label{E:J-1-combine}
a(e)\sigma_1(\widehat{e})(u_1(t(e))u_1(x)+u_2(t(e))u_2(x))(\phi(x)-\phi(t(e)))^2.
\end{equation}
In the sum $J_2$, the contribution of $e=[x,y]$ and $\widehat{e}=[y,x]$ together is
\begin{equation}\label{E:J-2-combine}
a(e)\sigma_2(\widehat{e})(-u_1(x)u_2(t(e))+u_1(t(e))u_2(x))(\phi(x)-\phi(t(e)))^2.
\end{equation}

Using~(\ref{E:J-1-combine}) and~(\ref{E:J-2-combine}), we can rewrite~(\ref{E:split-1-2}) to get~(\ref{E:H-leibniz-3}).
\end{proof}

We now give the definitions of minimal and maximal operators associated with the expression~(\ref{E:magnetic-schro}).

\subsection{Operators $\hmin$ and $\hmax$}\label{SS:operators-min-max} We define the operator $\hmin$ by the formula
\begin{equation}\label{E:h-min}
\hmin u:=Hu,\qquad \Dom(\hmin):=\vcomp.
\end{equation}

Since $q$ is real-valued, the following lemma follows easily from Lemma~\ref{L:symmetric}.

\begin{lem}\label{L:hvmin-symm} The operator $\hmin$ is symmetric in $\lw$.
\end{lem}
We define $\hmax:=(\hmin)^{*}$, where $T^*$ denotes the adjoint of operator $T$. We also define $\mathcal{D}:=\{u\in \lw\colon Hu\in \lw\}$.

\begin{lem} The following hold: $\Dom(\hmax)=\mathcal{D}$ and $\hmax u=Hu$ for all $u\in\mathcal{D}$.
\end{lem}
\begin{proof}
Suppose that $v\in\mathcal{D}$. Then, for all $u\in\vcomp$ we have
\[
(\hmin u,v)=(\Delta_{\sigma}u+qu,v)=(u, \Delta_{\sigma}v+qv).
\]
Since $(\Delta_{\sigma}v+qv)\in\lw$, by the definition of the adjoint we obtain $v\in\Dom((\hmin)^{*})$ and $(\hmin)^{*}v=\Delta_{\sigma}v+qv$. This shows that $\mathcal{D}\subset \Dom((\hmin)^{*})$ and $(\hmin)^*v=Hv$ for all $v\in\mathcal{D}$.

Suppose that $v\in\Dom((\hmin)^{*})$. Then, there exists $z\in\lw$ such that
\begin{equation}\label{E:h-min-adj}
(\Delta_{\sigma}u+qu,v)=(u, z),\qquad\textrm{for all }u\in\vcomp.
\end{equation}
Since $(\Delta_{\sigma}u+qu,v)=(u, \Delta_{\sigma}v+qv)$ and since $\vcomp$ is dense in $\lw$, from~(\ref{E:h-min-adj}) it follows that
$\Delta_{\sigma}v+qv=z=(\hmin)^*v$. This shows that $\Dom((\hmin)^*)\subset\mathcal{D}$. Thus,  we have shown that $\mathcal{D}= \Dom((\hmin)^{*})$ and $(\hmin)^*v=Hv$ for all $v\in\mathcal{D}$.
\end{proof}

\section{Proof of Theorem~\ref{T:main-0}}
We begin with a version of Kato's inequality for discrete magnetic Laplacian. For the original version in the setting of differential operators, see Kato~\cite{Kato72}. In the case $w\equiv 1$ and $a\equiv 1$, the following lemma was proven in Dodziuk--Mathai~\cite{Dodziuk-Mathai-03}.

\begin{lem}\label{L:kato-discrete} Let $\Delta$ and $\Delta_{\sigma}$ be as in~(\ref{E:ord-lap}) and~(\ref{E:magnetic-lap}) respectively. Then,  the following pointwise inequality holds for all $u\in C(V)$:
\begin{equation}\label{E:kato-discrete}
|u|\cdot\Delta|u|\leq \textrm{Re }(\Delta_{\sigma}u\cdot\overline{u}),
\end{equation}
where $\textrm{Re }z$ denotes the real part of a complex number $z$.
\end{lem}
\begin{proof}
Using~(\ref{E:ord-lap}), ~(\ref{E:magnetic-lap}), and the property $|\sigma(\widehat{e})|\leq 1$, we obtain
\begin{align*}
&(|u|\cdot\Delta|u|)(x)-\textrm{Re }(\Delta_{\sigma}u\cdot\overline{u})(x)\nonumber\\
&=\frac{1}{w(x)}\sum_{e\in\mathcal{O}_{x}}a(e)
\textrm{Re}(\sigma(\widehat{e})u(t(e))\overline{u(x)}-|u(x)||u(t(e))|)\leq 0,
\end{align*}
and the lemma is proven.
\end{proof}
\noindent\textbf{Continuation of the Proof of Theorem~\ref{T:main-0}.} Without loss of generality, we may assume $w(x)\equiv w_0=1$. By adding a constant to $q$, we may assume that $q(x)\geq 1$, for all $x\in V$. Let $\hmin$ and $\hmax$ be as in Section~\ref{SS:operators-min-max}.

Since $\hmin=H|_{\vcomp}$ is symmetric and since $(\hmin u,u)\geq \|u\|^2$, for all $u\in\vcomp$, the essential self-adjointness of $\hmin$ is equivalent to the following statement: $\ker(\hmax)=\{0\}$; see Reed--Simon~\cite[Theorem X.26]{rs}.  Let $u\in\Dom(\hmax)$ satisfy $\hmax u=0$:
\begin{equation}\label{E:zero-vector}
(\Delta_{\sigma}+q)u=0.
\end{equation}
By~(\ref{E:kato-discrete}) and~(\ref{E:zero-vector}) we get the pointwise inequality
\begin{equation}\label{E:kato-appl}
|u|\cdot\Delta|u|\leq \textrm{Re }(\Delta_{\sigma}u\cdot\overline{u})=\textrm{Re }(-qu\cdot\overline{u})=-q|u|^2\leq -|u|^2.
\end{equation}
Rewriting~(\ref{E:kato-appl}) we obtain the pointwise inequality
\begin{equation}\nonumber
|u|(\Delta|u|+|u|)\leq 0,
\end{equation}
which leads to
\begin{equation}\label{E:new-working}
0\geq (\Delta|u|)(x)+|u(x)|=\sum_{e\in\mathcal{O}_{x}}a(e)(|u(x)|-|u(t(e))|)+|u(x)|,
\end{equation}
for all $x\in V$.

From here on, the argument is the same as in Torki-Hamza~\cite[Theorem 3.1]{Torki-10}. Assume that there exists $x_0\in V$ such that $|u(x_0)|>0$. Then, by~(\ref{E:new-working}) with $x=x_0$, there exists
$x_1\in V$ such that $|u(x_0)|<|u(x_1)|$. Using~(\ref{E:new-working}) with $x=x_1$, we see that there exists $x_2\in V$ such that $|u(x_2)|>|u(x_1)|$. Continuing like this, we get a strictly increasing sequence of positive real numbers $|u(x_n)|$. But this contradicts the fact that $|u|\in \lw$. Hence, $|u|\leq 0$ for all $x\in V$. In other words, $u=0$. $\hfill\square$

\section{Proof of Theorem~\ref{T:main-1}}
In what follows, we will use a sequence  of cut-off functions.
\subsection{Cut-off functions}\label{SS:cut-off}
Fix a vertex $x_0\in V$, and  define
\begin{equation}\label{E:cut-off}
\phi_n(x):=\left(\left(\frac{2n-r(x)}{n}\right)\vee 0\right)\wedge 1,\qquad x\in V,\quad n\in \mathbb{Z}_{+},
\end{equation}
where  and $r(x)=d(x_0,x)$ is as in Section~\ref{SS:setting}.

As shown in Masamune~\cite[Proposition 3.2]{Masamune-09}, the sequence $\{\phi_n\}_{n\in\mathbb{Z}_{+}}$ satisfies the following properties:

\medskip

\noindent (i) $0\leq \phi_n(x)\leq 1$, for all $x\in V$;

\medskip

\noindent (ii)  $\phi_n(x)=1$ for $x\in B_{n}(x_0)$, and  $\phi_n(x)=0$ for $x\notin B_{2n}(x_0)$;

\medskip

\noindent (iii) $\displaystyle\sup_{e\in E_{s}} |(d\phi_n)(e)|\leq \frac{1}{n}$.

\bigskip

\noindent\textbf{Continuation of the Proof of Theorem~\ref{T:main-1}.} We will use a technique similar to Shubin~\cite[Section 5.1]{sh01}, Masamune~\cite{Masamune-09}, and Torki-Hamza~\cite{Torki-10}.

Since $H$ satisfies~(\ref{E:assumption-semibounded}), without loss of generality, we may add $(C+1)I$ to $H$ and assume that
\begin{equation}\label{E:assumption-semibounded-1}
    (Hv,v)\geq \|v\|^2,\qquad\textrm{ for all }v\in\vcomp.
\end{equation}
Since $\hmin=H|_{\vcomp}$ is symmetric and satisfies~(\ref{E:assumption-semibounded-1}), the essential self-adjointness of $\hmin$ is equivalent to the following statement: $\ker(\hmax)=\{0\}$; see Reed--Simon~\cite[Theorem X.26]{rs}.

Let $u\in\Dom(\hmax)$ satisfy $\hmax u=0$. Let $\phi_n$ be as in Section~\ref{SS:cut-off}.
Starting from~(\ref{E:H-leibniz-3}) with $\phi=\phi_n$ and using the properties (ii) and (iii) of $\phi_n$, together with  $|\sigma_1|\leq 1$ and $|\sigma_2|\leq 1$, we get the following estimate:
\begin{align}\label{E:H-leibniz-4}
&(H(u\phi_n), u\phi_n)\nonumber\\
&\leq\frac{1}{n^2}\sum_{e\in B_{2n}(x_0)}a(e)(u_1^2(t(e))+u_1^2(o(e))+u_2^2(t(e))+u_2^2(o(e))),
\end{align}
where $B_{2n}(x_0)$ is as in property (ii) of $\phi_n$.

By~(\ref{E:bounded-assumption-2}) and~(\ref{E:H-leibniz-4}) we obtain
\begin{align}\label{E:H-leibniz-5}
(H(u\phi_n), u\phi_n)&\leq\frac{m_{2n}a_{2n}}{n^2}\sum_{x\in B_{2n}(x_0)}w(x)((u_1(x))^2+(u_2(x))^2)\nonumber\\
&\leq\frac{m_{2n}a_{2n}}{n^2}\|u\|^2.
\end{align}
Since $\phi_nu\in\vcomp$, the inequality~(\ref{E:assumption-semibounded-1}) is satisfied with $v=\phi_n u$. Combining~(\ref{E:H-leibniz-5}) and~(\ref{E:assumption-semibounded-1}) we get
\begin{equation}\label{E:final-1}
    \|u\phi_n\|^2\leq \frac{m_{2n}a_{2n}}{n^2}\|u\|^2.
\end{equation}
We now take the limit as $n\to \infty$ in~(\ref{E:final-1}). Using the assumption~(\ref{E:bounded-assumption}) and the definition of $\phi_n$, we obtain $\|u\|^2\leq 0$. This shows that $u=0$. $\hfill\square$

\section{Proof of Theorem~\ref{T:main-2}}
In the case $w\equiv 1$, the following family of cut-off functions was constructed in Torki-Hamza~\cite{Torki-10}.

\subsection{Family of cut-off functions}\label{SS:family-cutoff}
Fix $x_0\in V$. For $R>0$ define
\begin{equation}\label{E:nbd-u}
U_{R}:=\{x\in V\colon d_{w,a}(x_0,x)\leq R\},
\end{equation}
where $d_{w,a}$ is as in~(\ref{E:w-a-dist}).
Define
\begin{equation}\label{E}
    \psi_{R}:=\min\{1,d_{w,a}(x,\,V\setminus U_{R+1})\}.
\end{equation}
The family $\psi_{R}$ satisfies the following properties:

\medskip

\noindent (i) $\psi_{R}(x)\equiv 1$, for all $x\in U_{R}$; (ii) $\psi_{R}(x)\equiv 0$, for all $x\in V\setminus U_{R+1}$; (iii) $0\leq \psi_{R}\leq 1$, for all $x\in V$;

\medskip

\noindent (iv) $\psi_{R}$ has finite support;

\medskip

\noindent (v) $\psi_{R}$ is a Lipschitz function with Lipschitz constant $1$.

\medskip

It is easy to see that the properties (i), (ii), (iii) and (v) hold. To prove property (iv), we will show that $U_{R+1}$ is finite. Clearly, $U_{R+1}$ is a closed and bounded set. With $d_{w,a}$ defined as in~(\ref{E:w-a-dist}), it follows that  $(V,d_{w,a})$ is a length space in the sense of Burago--Burago--Ivanov~\cite[Section 2.1]{burago-book}. Additionally, we know by hypothesis that $(V,d_{w,a})$ is complete. Thus,  by~\cite[Theorem 2.5.28]{burago-book} the set $U_{R+1}$ is compact. Suppose that there exists a sequence of vertices $\{x_n\}_{n\in\mathbb{Z}_{+}}\subset U_{R+1}$. Since $U_{R+1}$ is compact, there exists a subsequence, which we again denote by $\{x_n\}_{n\in\mathbb{Z}_{+}}$, such that $x_n\to x$ and $x\in U_{R+1}$.  Let $F=\{y_{1},\,y_{2},\,\dots,\,y_{s}\}$ be the set of all vertices $y\in V$ such that there is an edge connecting $y$ and $x$. The set $F$ is finite since $G$ is locally finite. Let $k_0=\max\{n\colon x_n\in F\}$ (if there is no $x_n$ such that  $x_n\in F$, we take $k_0=0$). Take $\epsilon>0$ such that $\displaystyle\epsilon<\min_{1\leq j\leq s}(d_{w,a}(y_{j},x))$. Then there exists $n_0\in \mathbb{Z}_{+}$ such that $d_{w,a}(x_n,x)<\epsilon$ for all $n\geq n_0$. Take $K\in\mathbb{Z}_{+}$ such that $K>\max\{k_{0}, n_0\}$. Clearly, $d_{w,a}(x_{K},x)<\epsilon$. Since $(V, d_{w,a})$ is a complete locally compact length space, by~\cite[Theorem 2.5.23]{burago-book} there is a shortest path $\gamma$ connecting $x_{K}$ and $x$. This means that the length $L_{w,a}(\gamma)$ of the path $\gamma$ satisfies
\begin{equation}\label{E:path-length-gamma}
L_{w,a}(\gamma)=d_{w,a}(x_{K},x)<\epsilon.
\end{equation}
Since $x_{K}\notin F$, there is no edge connecting $x_{K}$ and $x$. Hence, the path $\gamma$ will contain a vertex $y_{j}\in F$. Thus, $L_{w,a}(\gamma)>d_{w,a}(y_{j},x)>\epsilon$, and this contradicts~(\ref{E:path-length-gamma}). Hence, the set $U_{R+1}$ is finite.

\medskip

\noindent\textbf{Continuation of the Proof of Theorem~\ref{T:main-2}.} We adapt the technique of Torki-Hamza~\cite{Torki-10} to our setting.

As in the proof of Theorem~\ref{T:main-1}, without the loss of generality, we will assume~(\ref{E:assumption-semibounded-1}) and show that $\ker(\hmax)=\{0\}$. Let $u\in\Dom(\hmax)$ satisfy $\hmax u=0$. Using~(\ref{E:H-leibniz-3}) with $\phi=\psi_{R}$, we get
\begin{align}\label{E:H-leibniz-4-psi}
&(H(u\psi_{R}), u\psi_{R})\nonumber\\
&=\frac{1}{2}\sum_{x\in V}\sum_{e\in\mathcal{O}_{x}}a(e)\sigma_1(\widehat{e})[u_1(t(e))u_1(o(e))+\nonumber\\
&+u_2(t(e))u_2(o(e))](\psi_{R}(o(e))-\psi_{R}(t(e)))^2\nonumber\\
&+\frac{1}{2}\sum_{x\in V}\sum_{e\in\mathcal{O}_{x}}a(e)\sigma_2(\widehat{e})[-u_1(o(e))u_2(t(e))+\nonumber\\
&+u_1(t(e))u_2(o(e))](\psi_{R}(o(e))-\psi_{R}(t(e)))^2,
\end{align}
where $\mathcal{O}_{x}$ is as in~(\ref{E:neighborhood-o-x}).
Using the inequality $2\alpha\beta\leq \alpha^2+\beta^2$, properties $|\sigma_1|\leq 1$ and $|\sigma_2|\leq 1$, and the invariance of $a(e)$ and
\[
(u^2_{j}(t(e))+u^2_{j}(o(e)))(\psi_{R}(o(e))-\psi_{R}(t(e)))^2,\qquad j=1,2
\]
under involution $e\mapsto\widehat{e}$, we get
\begin{align}\label{E:H-leibniz-5-psi}
&(H(u\psi_{R}), u\psi_{R})\nonumber\\
&\leq\frac{1}{2}\sum_{x\in V}\sum_{e\in\mathcal{O}_{x}}a(e)(u^2_1(o(e))+u^2_2(o(e)))
(\psi_{R}(o(e))-\psi_{R}(t(e)))^2.
\end{align}

Using properties (i), (ii) and (v) of $\psi_{R}$,~(\ref{E:H-leibniz-5-psi}) leads to
\begin{eqnarray}\label{E:H-leibniz-6-psi}
&&(H(u\psi_{R}), u\psi_{R})\nonumber\\
&&\leq\frac{1}{2}\sum_{x\in U_{R+1}\setminus U_{R}}\,\sum_{e\in\mathcal{O}_{x}}a(e)|u(o(e))|^2
(d_{w,a}(o(e),t(e)))^2.
\end{eqnarray}
By~(\ref{E:w-a-dist}) and~(\ref{E:length-a-dist}) it follows that
\begin{equation}\label{E:H-leibniz-8-psi}
d_{w,a}(o(e),t(e))\leq \sqrt{\frac{w(o(e))}{a(e)}}.
\end{equation}
Using~(\ref{E:H-leibniz-6-psi}),~(\ref{E:H-leibniz-8-psi}), and bounded degree assumption on $G$, we get
\begin{equation}\label{E:H-leibniz-7-psi}
(H(u\psi_{R}), u\psi_{R})\leq\frac{N}{2}\sum_{x\in U_{R+1}\setminus U_{R}}\,w(x)|u(x)|^2.
\end{equation}
By property (iv) of $\psi_{R}$, it follows that $\psi_Ru\in\vcomp$; hence, the inequality~(\ref{E:assumption-semibounded-1}) is satisfied with $v=\psi_R u$. Combining~(\ref{E:H-leibniz-7-psi}) and~(\ref{E:assumption-semibounded-1}) we get
\begin{equation}\label{E:final-4}
    \|u\psi_{R}\|^2 \leq\frac{N}{2}\sum_{x\in U_{R+1}\setminus U_{R}}\,w(x)|u(x)|^2.
\end{equation}
We now take the limit as $R\to \infty$ in~(\ref{E:final-4}). Using the definition of $\psi_R$ and the assumption $u\in \lw$, we obtain $\|u\|^2\leq 0$. This shows that $u=0$. $\hfill\square$

\end{document}